\documentclass{article}[12pt]

\usepackage{latexsym}
\usepackage{amsmath}
\usepackage{amsbsy}
\usepackage{amsthm}
\usepackage{amssymb}
\usepackage{amsfonts}
\usepackage{dsfont}

\newtheorem{theorem}{Theorem}[section]

\def\E{{\mathds E}}
\def\I{{\mathds 1}}
\def\N{{\mathds N}}
\def\P{{\mathds P}}
\def\R{{\mathds R}}

\def\cS{{\mathcal S}}

\def\a{\alpha}
\def\b{\beta}
\def\d{\delta}

\def\s{\sigma}

\def\O{{E}}
\def\bin{{n}}

\begin{document}

\title{Poisson point processes: Large deviation inequalities for the convex distance}
\author{Matthias Reitzner\thanks{E-mail address: matthias.reitzner@uni-osnabrueck.de}\\Universit\"at Osnabr\"uck} 
\date{ }
\maketitle

\begin{abstract}
An analogue of Talagrand's convex distance for binomial and Poisson point processes is defined. A corresponding large deviation inequality is proved.
\end{abstract}

\section{Introduction and statement of results}
Since many years concentration of measure and large deviation inequalities are a subject of active research. Apart from theoretical interest much additional interest in these questions comes from applications in combinatorial optimization, stochastic geometry, and many others. For these problems a deviation inequality due to Talagrand \cite{Tal} turned out to be extremely useful. It combines the notion of {\it convex distance} with an elegant proof of a corresponding  dimension free deviation inequality. 

Let $(\O, \mathcal E, \P)$ be a probability space. Choose $n$ points $x_1, \dots, x_n \in \O$, $ x =(x_1, \dots, x_n) \in \O^n$, and assume that $ A \subset \O^n$ is measurable. 
Talagrand defines his convex distance by
\begin{equation}\label{def:dT}
d_T (x ,A)  = 
\sup_{ \| \a \|_2=1}\  \inf_{ y \in A} 
\sum_{1 \leq i \leq n} \a_i  \I(x_i \neq y_i) 
\end{equation}
where $\a=(\a_1, \dots , \a_n)$ is a vector in $S^{n-1}$. 
For $A \subset \O^n$ denote by 
$A_s:= \{ x: d_T(x, A) \leq s \} $ 
the $s$-parallel set of $A$ with respect to the convex distance.
Talagrand proves that for all $n \in \N$
\begin{equation}\label{eq:LDITal}
 \P^{\otimes n} (X \in A) \P^{\otimes n}  \left( X \notin A_s \right) \leq  e^{- \frac {s^2} 4}  
\end{equation}
where $X=(X_1, \dots , X_n)$ is a random vector with iid random variables $X_1, \dots, X_n$.

\bigskip
To extend this to point processes denote by $\bar  N(\O)$ the set of all 
finite counting measures $\xi= \sum_1^k \d_{x_i}$, $x_i \in \O$, $k \in \N_0$, or equivalently 
finite point sets $\{ x_1, x_2, \dots, x_k\}$ eventually with multiplicity. 

For a function $\a: \O \to \R$ we denote by $\| \a\|_{2,\xi}$ the 2-norm of $\a$ with respect to the measure $\xi$.
For two counting measures $\xi$ and $ \nu$ the (set-)difference $\xi \backslash \nu$  is defined by
\begin{equation} \label{def:xi-nu}
\xi \backslash \nu = 
\sum_{x:\, \xi(x)>0} (\xi(x)-\nu(x))_+ \, \d_x .
\end{equation}
As will be shown in Section \ref{sec:dT}, the natural extension of $d_T$ to 
Poisson counting measures $\eta \in \bar N(\O)$ with $\eta(\O) < \infty$ acts on $\bar  N(\O)$ and is defined by 
\begin{equation} \label{def:dTpi}
d_T^\pi (\eta ,A)  
=
\sup_{\| \a \|_{2,\eta}^2 \leq 1}\   \inf_{ \nu \in A} 
 \int \a  \, d(\eta \backslash \nu) 
\end{equation}
for $A \subset \bar  N(\O)$. Here the supremum is taken over all nonnegative functions $\a: \O \to \R$.

The main result of this paper is an extension of Talagrand's isoperimetric inequality to Poisson point processes on lcscH (locally compact second countable Hausdorff) spaces. If $\eta$ is a Poisson point process then the random variable $\eta(A)$ is Poisson distributed for each set $A \subset \O$ and the expectation $\E \eta (A)$ is the intensity measure of the point process.
For $A \subset \O^n$ we denote by 
$A_s^\pi := \{ x: d_T^\pi(x, A) \leq s \} $ 
the $s$-parallel set of $A$ with respect to the convex distance $d_T^\pi$.

\begin{theorem}\label{th:LDIallg}
Let  $\O$ be a lcscH space and let $\eta$ be a Poisson point process on $\O$ with finite intensity measure $\E \eta(\O)< \infty$. Then for any measurable subset $A \subset \bar  N(\O)$ we have
$$ \P (\eta \in A) \P  \left( \eta \notin A_s^\pi  \right) \leq  e^{- \frac {s^2} 4}  . 
$$ 
\end{theorem}

It is the aim of this paper to stimulate further investigations on this topic.
Of high interest would be an extension of this theorem to the case of point processes of possible infinite intensity measure. On the way to such a result one has to extend the notion of convex distance to locally 
finite counting measures with $\xi(\O)=\infty$. It is unclear whether (\ref{def:dTpi}) is the correct way to define convex distance in general, see the short discussion in Section~\ref{sec:dT}.

Our method of proof consists of an extension of Talagrand's large deviation inequality first to binomial processes and then to Poisson point processes. It would also be of interest to give a proof of our theorem using only methods from the theory of point processes. We have not been able to find such a direct proof for Theorem \ref{th:LDIallg}.

Our investigations are motivated by a problem in stochastic geometry. In \cite{ReSchuTh} Theorem \ref{th:LDIallg} is used to prove a large deviation inequality for the length of the Gilbert graph.

\section{Binomial point processes}

Assume that $\mu$ is a probability measure on $\O$. The sets $A,B $ considered in the following are measurable.

Let $\xi_n$ be a binomial point process on $\O$ of intensity $t \mu$, $ 0 \leq t \leq 1$, with parameter $n$. This means that for any $B \subset \O$ we have
\begin{equation}\label{def:binpro}
\P( \xi_n (B) = k) = 
\binom n k ( t \mu(B) )^{k} (1- t \mu(B) )^{n-k} .
\end{equation}

To link $\xi_n$ to $n$ iid points in $\O$ we consider the following natural construction.
We choose $n$ independent points in $\O$ according to the underlying probability measure $\mu $ and for each point we decide independently with probability $t$ if it occurs in the process or not.
To make this precise we add to $\O$ an artificial element $\triangle$ at infinity (containing of all points which have been deleted), define 
$\hat \O= \O \cup \{\triangle\}$ and extend $\mu$ to $\hat \O$ by 
$$\hat \mu (\hat B)= t \mu(\hat B \backslash \triangle) + \left(1- t\right) \,  \d_\triangle (\hat B) \  \mbox{ for } \hat B \subset \hat \O . $$
Hence a random point $X_i \in \hat \O$ chosen according to $\hat \mu$ is in $\O$ with probability $t$ and equals $\triangle$ with probability $1-t$. 
Define the projection $\pi$ of $x \in \hat \O^n  $ unto $ \bar  N(\O)$ by \lq deleting\rq\   all points $x_i=\triangle$, i.e.
$$
 \pi (x) = \pi ((x_1, \dots, x_n)) = \sum_{i=1}^n \I(x_i \neq \triangle) \d_{x_i} \in \bar  N(\O) .
$$
If $B \subset \O$, any set of $n$ iid random points $X_1, \dots, X_n$ chosen according to $\hat \mu$ satisfies
$$
\P \Big( \pi(X_1, \dots, X_n)(B) =k \Big)=
\binom n k ( t \mu(B) )^{k} (1- t \mu(B) )^{n-k}  .
$$
By the definition (\ref{def:binpro}) this shows that $\pi(X_1, \dots, X_n)(B)$ equals in distribution $\xi_n(B)$ for all sets $B \subset \O$. It is well known that this implies that $ \pi(X_1, \dots, X_n)$ equals in distribution $\xi_n$ for all subsets of $\bar N(\O)$.

Assume that $\hat A \subset \hat \O^n$ is a symmetric set, i.e. if $y=(y_1, \dots y_n) \in \hat A $ then also $(y_{\s(1)}, \dots, y_{\s(n)} ) \in \hat A$ for all permutations $\s \in  \cS_{n}$. Here $\cS_I$ is the group of permutations of a set $I \subset \N$, and we write $\cS_n$ if  $I=\{1, \dots, n\}$. It is immediate that a symmetric set is the preimage of a set $A \subset \bar N(\O)$ under the projection $\pi$ where 
$\pi(\hat A) = \bigcup_{y \in \hat A} \pi(y) \subset \bar N(\O) $.
As shown above for a random vector $X=(X_1, \dots, X_n)$ with iid coordinates we have 
\begin{equation} \label{eq:Pgleich}
 \P( X \in \hat A)= \P (\pi(X) \in \pi(\hat A)) = \P(\xi_n \in A) . 
\end{equation}

\medskip
The essential observation is that the convex distance $d_T( x, \hat A)$ defined in (\ref{def:dT}) for $x \in \hat \O^n$ is compatible with the projection $\pi$ and yields the convex distance 
\begin{eqnarray} \label{def:dTbin}
 d_T^{\bin} (\xi_n ,A)  &=& 
\sup_{\| \a\|_{2, \xi_n}^2  \leq 1 } \ \inf_{ \nu \in A } 
\Big[ \int \a d(\xi_n\backslash \nu) 
\\ \nonumber &&  \hskip3cm
+  \frac {(\nu(\O)- \xi_n(\O) )_+}{(n- \xi_n(\O))^{\frac 12}} (1- \| \a \|_{2, \xi_n}^2)^{\frac 12}  
\Big]  
\end{eqnarray}
on the space $\bar N(\O)$. 
\begin{theorem}\label{th:dt=dtbin}
Assume $x \in \hat \O^n$ and that $\hat A \subset \hat \O^n$ is a symmetric set. Then for $\xi_n = \pi(x)$ and $A = \pi(\hat A)$ we have 
$$d_T(x, \hat A)= d_T^{\bin} (\xi_n,  A)  . $$  
\end{theorem}
\begin{proof}
Since $\hat A$ is a symmetric set for any function $f$ 
$$
\inf_{ y \in \hat A} f(y_1, \dots , y_n) = 
\inf_{ y \in \hat A, \s \in \cS_n} f(y_{\s(1)}, \dots , y_{\s(n)}) ,
$$
and we can rewrite the convex distance on $\hat \O^n$ given by (\ref{def:dT}) as
\begin{eqnarray*}
d_T (x ,\hat A)  
&=& 
\sup_{\| \a\|_2\leq 1}\  \inf_{ y \in \hat A } \inf_{\s \in \cS_n} 
\sum_{1 \leq i \leq n} \a_i  \I( x_i  \neq y_{\s(i)}) .
\end{eqnarray*}
We write $\xi_n=\pi(x), \nu=\pi(y)$.
It is immediate by the symmetry of $\hat A$ that $d_T(x, \hat A)$ is invariant under any permutation of $x_1, \dots x_n$.
Hence we assume w.l.o.g. that $x_i $ are sorted in such a way that $x_i \neq \triangle$ for $i=1, \dots, \xi_n(\O)$ and $x_i=\triangle$ for $i \geq \xi_n(\O)+1$. 
$$
d_T (x ,\hat A)  \nonumber
= 
\sup_{\| \a\|_2 = 1}\  \inf_{ y \in \hat A } \inf_{\s \in S_{n}} 
\Big[ \sum_{i=1}^{ \xi_n(\O)} \a_i  \I( x_i  \neq y_{\s(i)})
+ \sum_{i= \xi_n(\O) +1}^{n}  \a_i \I(y_{\s(i)} \neq \triangle) 
\Big]  
$$
Here the second summand is zero if $\nu(\O) \leq \xi(\O)$.
For fixed $x$ and $y$ we decrease the summands if we assume that the permutation acts in such a way that the maximal number of $\triangle$'s in $x$ and $y$ coincide.
If $\nu (\O) \leq \xi_n(\O)$ this means that the minimum over $\cS_n$ is attained if $y_{\s(i)}= \triangle$ for all $\s(i) \geq \xi_n(\O)$ which coincides with the fact that the second summand in this case vanishes. If $\nu (\O) > \xi_n(\O)$ then $y_{\s(i)}= \triangle$ implies $\s(i) \geq \xi_n (\O)$. To make things more visible we take in this case the infimum over additional permutations $\tau \in \cS_{[\xi_n(\O)+1, n]}$ of the second summand.
\begin{eqnarray}\label{eq:dTsplit}
d_T (x ,\hat A)
&=& 
\sup_{\| \a\|_2 = 1}\  \inf_{ y \in \hat A } \inf_{\s \in S_{n}} 
\Big[ \sum_{i=1}^{ \xi_n(\O)} \a_i  \I( x_i  \neq y_{\s(i)})
\\ && \hskip3cm \nonumber
+ \inf_{\tau \in S_{[\xi_n(\O)+1, n]}}  \sum_{i= \xi_n(\O) +1}^{n}  \a_i \I(y_{\tau (\s(i))} \neq \triangle) 
\Big]  .
\end{eqnarray}
The second summand equals the sum of the $(\nu(\O) - \xi_n(\O))_+$ smallest $\a_i$'s in 
$\{ \a_{\xi_n(\O)+1} , \dots, \a_n \}$. We set $\a_{ \xi_n(\O)+i} =\b_i$ for $i=1, \dots, n-\xi_n(\O)$ and denote by $\b_{(1)} \leq \dots \leq \b_{(n- \xi_n(\O))}$ the order statistic of the $\b_i$. 
We obtain
\begin{eqnarray*}
d_T (\xi ,A)  
&=&
\sup_{\| \a\|_2^2 + \| \b \|_2 ^2 = 1} \ 
\inf_{ y \in \hat A } \inf_{\s \in S_{n}} 
\Big[ \sum_{i=1}^{ \xi_n(\O)} \a_i  \I( x_i  \neq y_{\s(i)})
+  \sum_{j= 1}^{(\nu(\O)- \xi_n(\O) )_+}  \b_{(j)} 
\Big]  
\end{eqnarray*}
where from now on  $\| \a \|_2^2 = \sum_1^{\xi_n(\O)} \a_i^2$.
The $\b_j^2$ sum up to $1- \| \a \|_2^2$ so that the sum of the $k$-th smallest is at most 
$  (1- \| \a \|_2^2) k/(n - \xi_n (\O))$. H\"older's inequality yields
\begin{eqnarray*}
d_T (x , \hat A)  
&\leq &
\sup_{\| \a\|_2^2 + \| \b \|_2 ^2 = 1} \ 
\inf_{ y \in \hat A } \inf_{\s \in S_{n}} 
\Big[ \sum_{i=1}^{ \xi_n(\O)} \a_i  \I( x_i  \neq y_{\s(i)})
\\ && \hskip3cm \nonumber
+ \Big( (\nu(\O)- \xi_n(\O))_+ \sum_{j= 1}^{(\nu(\O)- \xi_n(\O) )_+}  \b_{(j)}^2  \Big)^{\frac 12}
\Big] 
\\ &\leq &
\sup_{\| \a\|_2^2 \leq 1} \ 
\inf_{ y \in \hat A } \inf_{\s \in S_{n}} 
\Big[ \sum_{i=1}^{ \xi_n(\O)} \a_i  \I( x_i  \neq y_{\s(i)})
\\ && \hskip3cm \nonumber
+  \frac {(\nu(\O)- \xi_n(\O))_+}{(n - \xi_n (\O))^{\frac 12}} (1- \| \a \|_2^2)^{\frac 12} 
\Big] .
\end{eqnarray*}
On the other hand if we take $\b_j^2 = (1- \| \a \|_2^2)/(n- \xi_n(\O))$ we have 
$ \| \b \|_2^2 
=1 - \| \a \|_2^2 $
and the supremum is bounded from below by setting $\b_j$ equal to these values.
\begin{eqnarray*}
d_T (x , \hat A)  
&\geq &
\sup_{\| \a\|_2^2  \leq 1 } \ 
\inf_{ y \in \hat A } \inf_{\s \in S_{n}} 
\Big[ \sum_{i=1}^{ \xi_n(\O)} \a_i  \I( x_i  \neq y_{\s(i)})
\\ && \hskip3cm \nonumber
+  \frac {(\nu(\O)- \xi_n(\O) )_+}{(n- \xi_n(\O))^{\frac 12}} (1- \| \a \|_2^2)^{\frac 12}  
\Big]  . 
\end{eqnarray*}
Both bounds coincide so that $d_T$ equals the right hand side.
Define the function $\a: \O \to \R$ by
$$ \a (x)= \left\{ \begin{array}{ll}
                    \a_i \ & \mbox{ if } x=x_i \\
                    0 & \mbox{ otherwise}
                   \end{array} \right.
$$
so that 
$
\| \a \|_{2, \xi}^2= \int \a^2 d\xi  = \sum_{i=1}^{\xi_n(\O)} \a_i^2  = \| \a \|_2^2  
$
and by the definition (\ref{def:xi-nu}) of $\xi \backslash \nu$ 
$$
\inf_{\s \in \cS_{\xi_n(\O)}}  \sum_{i=1}^{\xi_n (\O) } \a_i  \I(x_i \neq y_{\s(i)}) 
= \int \a d(\xi\backslash \nu) .
$$
This proves
$$ d_T (x , \hat A)  = 
\sup_{\| \a\|_{2, \xi_n}^2  \leq 1 } \ \inf_{ \nu \in A } 
\Big[ \int \a d(\xi\backslash \nu) 
+  \frac {(\nu(\O)- \xi_n(\O) )_+}{(n- \xi_n(\O))^{\frac 12}} (1- \| \a \|_{2, \xi_n}^2)^{\frac 12}  
\Big]  . 
$$
\end{proof}

By (\ref{eq:Pgleich}) we have $\P(X \in \hat A)=\P(\xi_n \in A)$ 
for any measurable symmetric subset $\hat A$ of $\O^n$. Recall that $\xi_n = \pi(X)$ and $A = \pi(\hat A)$.  
Theorem (\ref{th:dt=dtbin}) shows that 
$$d_T(X, \hat A) \geq s \ \mbox{ iff }\  d_T^{\bin}(\xi_n, A) \geq s ,$$ 
so that $X\notin \hat A_s$ iff $\xi_n \notin A_s^{\bin}$. Here we denote by $A_s^{\bin}$ the parallel set with respect to the distance $d_T^{\bin}$.
Again by (\ref{eq:Pgleich}) this yields 
$\P(X \notin \hat A_s)  =\P(\xi_n \notin A_s^{\bin})$. 
Combining this with Talagrand's large deviation inequality (\ref{eq:LDITal}), 
$$
\P(X \in \hat A) \P(X \notin \hat A_s) \leq e^{- \frac{s^2}4}
$$
we obtain a large deviation inequality for the binomial process.
\begin{theorem}\label{th:LDIbin}
Assume $\xi_n$ is a binomial point process with parameter $n$ on $\O$. Then we have 
$$ \P (\xi_n \in A) \P  \left( \xi_n \notin A_s^{\bin}  \right) \leq  e^{- \frac {s^2} 4}  $$ 
for any $A \subset \bar N(\O)$.
\end{theorem}

\section{Poisson point processes}

We extend Theorem \ref{th:LDIbin} to Poisson point processes using the usual approximation of a Poisson point process by Binomial point processes. Assume that the state space $\O$ is a lcscH space (locally compact second countable Hausdorff space) and that $\mu$ is a probability measure on $\O$.

Fix some $t>0$ and recall that $\mu$ is a probability measure on $\O$. Set $t_n= t/n$ for $n \in \N$, $t \geq 0$ and assume that $n$ is sufficiently large such that $t/n \leq 1$.
It is well know (see Jagers \cite{Jag}, or Theorem 16.18 in Kallenberg \cite{Kallenberg2002}) that on an lcscH space $\O$  a sequence of binomial point processes $\xi_n$ with intensity measure $t_n \mu$ and parameter $n$ converge in distribution to a Poisson point process $\eta$ with intensity measure $t \mu$ as $n \to \infty$. 
Thus for $B \subset \bar N(\O)$ we have as $n \to \infty$
\begin{equation}\label{eq:bintoPois}
\P(\xi_n \in B) \to  \P (\eta \in B) .
\end{equation}

The distance $d_T^{\bin}$ depends on $n$ and has to be extended from binomial to Poisson point processes as $n \to \infty$. As stated in the introduction we use as a suitable definition
$$
 d_T^\pi (\xi, A)= 
\sup_{\| \a\|_{2, \xi_n}^2  \leq 1 } \ \inf_{ \nu \in A } \int \a d(\xi \backslash \nu) 
$$
This is motivated by the more detailed investigations in Section \ref{sec:dT}. 
For $A \subset \bar N(\O)$ let $A_s^\pi$ be the parallel set of $A$ with respect to the distance $d_T^\pi$.
It is immediate that 
\begin{equation} \label{eq:dTpi<bin}
d_T^\pi (\xi, A) \leq d_T^{\bin} (\xi, A) . 
\end{equation}
This implies
$ A_s^\pi \supset A_s^{\bin} $ and thus for a binomial point process $\xi_n$
\begin{equation}\label{eq:aspi<}
 \P(\xi_n \notin A_s^\pi) \leq \P(\xi_n \notin A_s^{\bin}) .
\end{equation}
The following theorem is an immediate consequence of Theorem \ref{th:LDIbin} and formulae (\ref{eq:aspi<}) and (\ref{eq:bintoPois}).
\begin{theorem}\label{th:LDIpoi}
Assume $\eta$ is a Poisson point process 
on some lcscH space $\O$ with $\E \eta (\O)< \infty$. Then we have 
$$ \P (\eta \in A) \P  \left( \eta \notin A_s^\pi  \right) \leq  e^{- \frac {s^2} 4}  $$ 
for any $A \subset \bar N(\O)$.
\end{theorem}

\section{The convex distance}\label{sec:dT}

In this section we collect some facts about the convex distances $d_T^{\bin}$ and $d_T^\pi$ on $\bar N(\O)$. To start with we show that $d_T^{\bin}$ not only gives the lower bound (\ref{eq:dTpi<bin}) for the distance $d_T^\pi$ defined in (\ref{def:dTpi}).
We also prove that $d_T^{\bin} \to d_T^\pi$ as $n \to \infty$ which shows that there is essentially no other natural choice for $d_T^\pi$.

We start with the representation (\ref{def:dTbin}) of the convex distance for binomial point processes. 
Assume $\xi \in \bar N(\O)$ satisfies $\xi(\O)< \infty$. Set
$$D_\a = \inf_{ \nu \in A } \Big[ \int \a d(\xi \backslash \nu) 
+  \frac {(\nu(\O)- \xi (\O) )_+}{(n- \xi (\O))^{\frac 12}} (1- \| \a \|_{2, \xi }^2)^{\frac 12}  
\Big]  
$$
so that 
$ d_T^{\bin} (\xi  ,A)  = 
\sup \{ D_\a:\ \| \a\|_{2, \xi }^2  \leq 1  \} $.

Since $\xi $ is finite, the map $\nu \to \xi  \backslash\nu $ can take only finitely many \lq values\rq\  $\mu_1, \dots, \mu_m \in \bar N(\O)$. Write $A_1, \dots , A_m \subset A$ for the preimage of the measures $\mu_i$ under this map. 
Denote for $i=1, \dots, m$ by $\nu_i $ one of the counting measures in $A_i$ with minimal $\nu(\O)$. Note that these minimizers are independent of $\a$. Assume that $\nu_i(\O) \leq \dots \leq \nu_m(\O)$.
We compute the infimum over $\nu \in A$ by taking the infimum over $\nu \in A_i$ and then the minimum over $i=1, \dots, m$.
\begin{eqnarray*}
D_\a
& = &
\min_{i=1, \dots, m}  
\left[ \int \a d\mu_i  + \frac {\inf_{\nu \in A_i} (\nu(\O)- \xi (\O) )_+}{(n- \xi (\O))^{\frac 12}} (1- \| \a \|_{2, \xi }^2)^{\frac 12}  
 \right] .
\\ & = &
\min_{i=1, \dots, m}  
\left[ \int \a d\mu_i  + \frac {(\nu_i(\O)- \xi (\O) )_+}{(n- \xi (\O))^{\frac 12}} (1- \| \a \|_{2, \xi }^2)^{\frac 12}  
 \right] .
\end{eqnarray*}
Since $\nu_i(\O)$ is bounded by $\nu_m(\O)$ we have 
\begin{eqnarray*}
\min_{i=1, \dots ,m} \int \a d\mu_i 
\leq
D_\a
\leq 
\min_{i=1, \dots ,m} \int \a d\mu_i  + \frac {(\nu_m(\O)- \xi (\O) )_+}{(n- \xi (\O))^{\frac 12}}  
\end{eqnarray*}
which shows that for $n \to \infty$ the distance $d_T^{\bin}$ converges to 
$$ d_T^\pi (\xi, A)= 
\sup_{\| \a\|_{2, \xi }^2  \leq 1 } \ \inf_{ \nu \in A } \int \a d(\xi \backslash \nu) . 
$$
Note that this is only a pseudo-distance since $d_T^\pi (\xi, A)=0$ does not imply that $\xi \in A$. For $d_T^\pi (\xi, A)=0$  it suffices that $A$ containes some counting measure of the form $\xi + \nu$ with $\nu \in \bar N(\O)$ because then $\xi \backslash \nu=0$.

It would be nice to have a definition of $d_T^\pi$ which is a distance for counting measures and which indicates extensions to point processes with possibly unbounded $\E \xi(E)$. 
One could also use the distance $d_T^\pi$ given in (\ref{def:dTpi}) as a definition but we could not relate it to the distance $d_T$ for binomial processes. 
In applications it would be of high importance to have such a representation and a large deviation inequality at least for Poisson point processes on $\R^d$. To the best of our knowledge even recent results like the one by Wu \cite{Wu} or Eichelsbacher, Raic and Schreiber \cite{ERS} cannot be easily extended to our setting.


\begin{thebibliography}{9}
\bibitem{ERS}
P. Eichelsbacher and M. Raic and T. Schreiber, 
{Moderate deviations for stabilizing functionals in geometric probablity}.
manuscript (arXiv:1010.1665) (2013)

\bibitem{Jag}
P. Jagers,
{On the weak convergence of superpositions of point processes}.
Z. Wahrsch. Verw. Geb. {\bf 22}, 1--7, (1972)

\bibitem{Kallenberg2002}
O. Kallenberg,
{Foundations of Modern Probability}, Second Edition.
Probability and its Applications (New York), Springer-Verlag, New York (2002)

\bibitem{ReSchuTh}
M. Reitzner and M. Schulte and C. Th\"ale,
{Limit theory for the Gilbert graph}.
manuscript

\bibitem{Tal}
M. Talagrand,
{Concentration of measure and isoperimetric inequalities in product spaces}.   
Inst. Hautes \'Etudes Sci. Publ. Math. {\bf 81}, 73--205 (1995)

\bibitem{Wu}
L. Wu,
{A new modified logarithmic Sobolev inequality for poisson point processes and several applications}.
Probab. Theory Relat. Fields {\bf 118}, 427--438 (2000)


\end{thebibliography}
\end{document}